\documentclass[11pt]{amsart}
\usepackage{amsfonts}
\usepackage{mathrsfs}
\usepackage{graphicx}
\usepackage{verbatim}
\usepackage[abs]{overpic}
\usepackage{amsmath,amsthm,amssymb}
\usepackage{geometry}
\newtheorem{theorem}{Theorem}[section]
\newtheorem{lemma}[theorem]{Lemma}

\newtheorem{proposition}[theorem]{Proposition}

\theoremstyle{definition}
\newtheorem{definition}[theorem]{Definition}

\theoremstyle{remark}
\newtheorem{remark}[theorem]{Remark}

\newtheorem{question}[theorem]{Question}
\newtheorem{construction}[theorem]{Construction}
\numberwithin{equation}{section}

\author{James Conway}
\address{Department of Mathematics \\ University of California, Berkeley}
\email{conway@berkeley.edu}

\author{Amey Kaloti}
\address{Department of Mathematics \\ University of Iowa}
\email{amey-kaloti@uiowa.edu}

\author{Dheeraj Kulkarni}
\address{Department of Mathematics \\ RKM Vivekananda University}
\email{dheeraj.kulkarni@gmail.com}

\begin{document}
\title[Planar Open Books and the Contact Invariant]{Tight Planar Contact Manifolds with Vanishing Heegaard Floer Contact Invariants}
\maketitle

\begin{abstract}
In this note, we exhibit infinite families of tight non-fillable contact manifolds supported by planar open books with vanishing Heegaard Floer contact invariants. Moreover, we also exhibit an infinite such family where the supported manifold is hyperbolic.
\end{abstract}


\section{Introduction}
Many techniques have been developed to determine whether a given contact $3$-manifold is tight.  As each one was developed, the question arose as to whether a given property is equivalent to tightness.  Fillability was the first widely-used tool to prove tightness, but tight contact manifolds were found by Etnyre and Honda that were not fillable \cite{Etnyre_Honda_nonfillable}.  Ozsv\'{a}th and Szab\'{o} then developed Heegaard Floer theory, which very promisingly could prove tightness in many cases where the manifold was not fillable.  However, many tight contact manifolds with vanishing Heegaard Floer contact invariant have been discovered, see \cite{Ghiggini2,Ghiggini1}.

On the side of positive results, Honda, Kazez, and Mati\'{c} \cite{HKM_contact_invariant} proved that for contact manifolds supported by open books with pages of genus one and connected binding, tightness is equivalent to the non-vanishing of the Heegaard Floer contact invariant.  For contact manifolds supported by open books with planar pages, the first tight but non-fillable examples came from applying \cite[Corollary 1.7]{LS3} to examples in \cite{GLS,LS2}.  They proved tightness by showing that the Heegaard Floer contact invariant does not vanish.  The following question, however, remained.

\begin{question} For contact manifolds supported by open books with planar pages, is tightness equivalent to the non-vanishing of the Heegaard Floer contact invariant? \end{question}

We provide a negative answer to this question.  In particular, we construct infinite families of tight contact manifolds supported by planar open books where capping off a binding component recovers a given overtwisted manifold.  Since capping off a binding component of an open book is equivalent to admissible transverse surgery on a binding component, we have the following.

\begin{theorem}
\label{Th:Main_Theorem}
Given any overtwisted contact manifold $(M,\xi)$, there exist infinite families of tight, planar, non-fillable contact manifolds with vanishing Heegaard Floer contact invariant, on which admissible transverse surgery on a link in these manifolds recovers $(M,\xi)$.
\end{theorem}

We hope that these examples will be useful in further exploring characterizations of tightness, and in particular will provide computable examples of boundary cases.  Indeed, these examples have already been used to investigate possible new invariants of contact structures coming from Heegaard Floer Homology being developed by Kutluhan, Mati\'{c}, Van Horn-Morris, and Wand \cite{KMVHMW2} and Baldwin and Vela-Vick \cite{BVV}.  These invariants are interesting when the Heegaard Floer contact invariant vanishes, as in the examples from Theorem~\ref{Th:Main_Theorem}.

The only infinite family of tight, non-fillable contact manifolds that are hyperbolic has been produced by Baldwin and Etnyre \cite{Baldwin_Etnyre_transverse}.  Their examples are supported by open books with pages of genus one.  In addition, their construction requires throwing out a finite number of unspecified members of their infinite family that may not be hyperbolic.  We produce an infinite family of such manifolds with planar supporting open books, all of whom are hyperbolic.  To this end, let $S$ denote the surface shown in Figure~\ref{fig:Planar_Surface_With_the_Curves_Drawn}. Let $\mathbf{v} = (p, n_1, n_2, n_3, n_4)$ be a $5$-tuple of integers, and let  $\phi_\mathbf{v}= \tau_{\alpha}^{-n_1-1} \tau_{\beta}^p \tau_{B_1}^{n_1} \tau_{B_2}^{n_2} \tau_{B_3}^{n_3} \tau_{B_4}^{n_4}$ be a diffeomorphism of $S$.  We show the following.

\begin{figure}[htb]
\begin{overpic}
{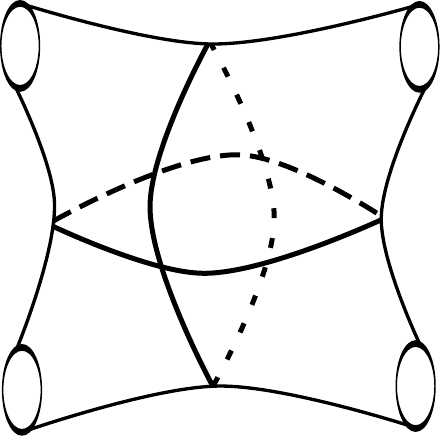}
\put(160, 80){$\alpha$}
\put(70,160){$\beta$}
\put(-14,190){$B_1$}
\put(214,190){$B_2$}
\put(-14,20){$B_4$}
\put(212,20){$B_3$}
\end{overpic}
\caption{Surface $S$ used in constructing $(S,\phi_\mathbf{v})$, with $\alpha$ and $\beta$ curves indicated.}
\label{fig:Planar_Surface_With_the_Curves_Drawn}
\end{figure}
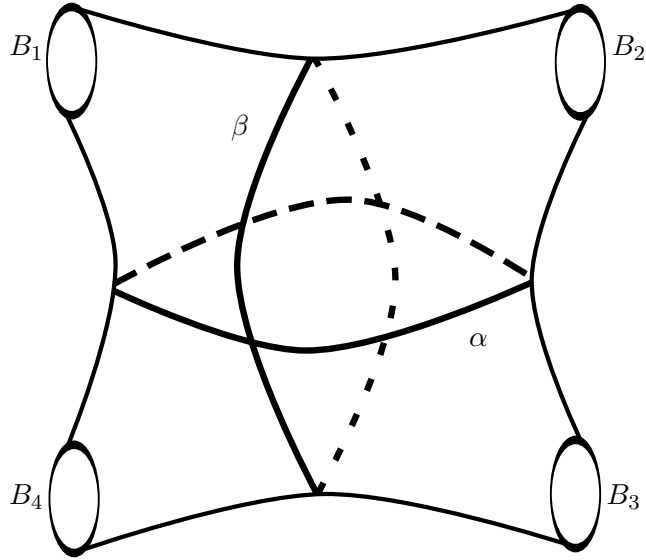

\begin{theorem}
\label{Th:hyperbolic_examples}
Let $(M_\mathbf{v},\xi_\mathbf{v})$ be the contact manifold supported by the open book $(S,\phi_\mathbf{v})$. Then $(M_{\mathbf{v}},\xi_\mathbf{v})$ is universally tight, not fillable, has vanishing Heegaard Floer contact invariant, and is hyperbolic, for $p \geq 1$, and $n_i \geq 6$ for each $i$.
\end{theorem}

This paper is organized as follows. In Section~\ref{Section:Background} we recall the definitions and properties of open books, fillability, the Heegaard Floer contact invariant, and transverse surgery. We prove our results in Section~\ref{Section:Proofs}.

\textbf{Acknowledgements:} The authors would like to thank John Etnyre for helpful conversations. The first author was supported in part by NSF grant DMS-1309073.  The second author was supported in part by NSF grant DMS-0804820.  The third author was supported by Indo-US Virtual Institute For Mathematical and Statistical Sciences (VIMSS) for visiting Georgia Tech; he would also like to thank the School of Mathematics at Georgia Tech for their hospitality during his visit.


\section{Background}
\label{Section:Background}
In this section we recall basic notions from contact geometry. See~\cite{Geiges_Contact_Geometry_Book} for an overview of the basics of contact geometry.

The fundamental dichotomy of contact manifolds is between tight and overtwisted.  A contact manifold $(M,\xi)$ is called \textit{overtwisted} if there exists an embedded disc $D$ in $M$ such that $\partial D$ is tangent to the contact planes and that the framing given by the contact planes agrees with the framing given by the surface $D$. If $(M,\xi)$ is not overtwisted, we call it \textit{tight}.

\subsection{Open Book Decompositions}
An \textit{abstract open book decomposition} is a pair $(S,\phi)$ with the following properties.
\begin{itemize}
\item $S$ is an oriented compact surface with boundary called the \textit{page} of the open book decomposition.
\item $\phi: S \rightarrow S$ is a diffeomorphism of $S$ such that $\phi|_{\partial S}$ is the identity. The diffeomorphism $\phi$ is called the \textit{monodromy} of the open book decomposition.
\end{itemize}

We construct a 3-manifold $M_{(S,\phi)}$ from the mapping torus of $(S,\phi)$ in the following way:
\[ M_{(S,\phi)} = \frac{S \times [0,1]}{(x,1) \sim (\phi(x),0), (x \in \partial S, t) \sim (x, t')}. \]

The image of boundary $\partial S \times \{0\}$ is a fibered link in $M$, and is called the \textit{binding} of the open book decomposition.  By abuse of notation, we call the boundary of the surface $S$ the binding as well.

Given a $3$-manifold $M$, an \textit{open book decomposition for $M$} is a diffeomorphism of $M$ with $M_{(S,\phi)}$, for some $(S,\phi)$.  Given a contact structure $\xi$ on $M$ we say that the open book decomposition $(S,\phi)$ \textit{supports} $\xi$ if $\xi$ is isotopic to a contact structure with a defining $1$-form $\alpha$ such that:
\begin{itemize}
\item $\alpha > 0$ on the binding $\partial S$.
\item $d\alpha$ is a positive area form on the interior of each page $S \times \{t\}$ of the open book decomposition.
\end{itemize}
In the supported contact structure, the binding becomes a transverse link.  Every open book supports a unique contact structure up to isotopy, and every contact structure has a supporting open book.  The following foundational result of Giroux stated in~\cite{Giroux_Correspondence} forms the bedrock of the interaction between open book decompositions and contact geometry.

\begin{theorem}[Giroux \cite{Giroux_Correspondence}]
\label{Giroux 1-1}
Let $M$ be a closed oriented 3-manifold. Then there is a one-to-one correspondence between oriented contact structures on $M$ up to contactomorphism and open book decompositions of $M$ up to \textit{positive (de)stabilization}, \textit{ie.\ }plumbing the page with an annulus and adding a positive Dehn twist to the monodromy along the core of the annulus, and the reverse process.
\end{theorem}

\begin{remark}
Giroux actually defined open books on a manifold $M$ by embedded fibered knots that are null-homologous.  This gives an abstract open book by considered the compactification of the fibers of the fibration, along with the induced monodromy.  In this setting, Giroux proved that the correspondence was up to isotopy of contact structures on $M$.
\end{remark}

Given a contact manifold $(M,\xi)$, we define its \textit{support genus} to be the minimum of the genus $g(S)$ over all open book decompositions $(S,\phi)$ supporting $(M,\xi)$.  A contact structure is \textit{planar} if its support genus is zero.  The following result of Etnyre describes the support genus of overtwisted contact structures.

\begin{theorem}[Etnyre \cite{Etnyre_Overtwisted_Planar}]
\label{Etnyre OT planar}
If $(M,\xi)$ is an overtwisted contact manifold, then $\xi$ is planar.
\end{theorem}

Given an open book $(S,\phi)$ with binding components $B_1,\ldots, B_k$, $k >1$, Baldwin \cite{Baldwin_capping_off} defines an operation called \textit{capping off}.  Given a boundary component $B_i$, we create a new surface $S'$ by gluing $D^2$ to $S$ along $B_i$.  We create a monodromy map $\phi'$ for $S'$ by extending $\phi$ over the glued-in $D^2$ by the identity map.  This gives a new open book decomposition $(S',\phi')$.

\subsection{Fractional Dehn Twist Coefficients}
Let $(S,\phi)$ be an open book decomposition with binding components $B_1,\ldots,B_k$.  By a result of Thurston \cite{Thurston_Classification}, the monodromy $\phi$ is freely isotopic to a map $\Phi$ that falls into one of three categories: pseudo-Anosov, periodic, or reducible.  We will define the \textit{fractional Dehn twist coefficient} of $\phi$ at $B_i$ where $\Phi$ is periodic or reducible; the definition for pseudo-Anosov maps is more involved, and as we will not need it, we will omit it here.  See \cite[Section 4]{Ito_Kawamuro_essential} for more details.

\begin{definition}
If $\Phi$ is periodic, then $\Phi^n$ is isotopic to $\tau_{B_1}^{m_1}\tau_{B_2}^{m_2}\cdots\tau_{B_k}^{m_k}$ for some positive integer $n$, and integers $m_1, \ldots, m_k$, where $B_1, \ldots, B_k$ are isotopic to the boundary components of $S$.  We define the \textit{fractional Dehn twist coefficient} of $\phi$ at $B_i$ to be $c(\phi,B_i) = m_i/n.$

If $\Phi$ is reducible, then there exists some simple closed multicurve on $S$ that is preserved by $\Phi$.  In this case, there is some maximal subsurface $S'$ of $S$ containing $B_i$ on which $\Phi|_{S'}$ is well defined and is either pseudo-Anosov or periodic.  On this subsurface, $\Phi|_{S'}$ may permute some boundary components.  Take a large enough power $n$ of $\Phi|_{S'}$ such that all boundary components are fixed. We define the \textit{fractional Dehn twist coefficient} of $\phi$ at $B_i$ to be $1/n$ times the fractional Dehn twist coefficient of $\Phi|_{S'}^n$ at $B_i$.  We will only consider reducible monodromies that restrict to being periodic.
\end{definition}

We will not calculate the fractional Dehn twist coefficient of a pseudo-Anosov monodromy from the definition, but we will use the following estimate, due to Kazez and Roberts.  First, we give the following definition, which will make the theorem easier to state.  Given an open book $(S,\phi)$ and a boundary component $B$ of $S$, we say that $\phi$ is {\it right-veering} at $B$ if, for every properly embedded arc $\alpha$ with an endpoint on $B$, the image $\phi(\alpha)$ is either isotopic to $\alpha$ or to an arc right of $\alpha$ near $B$ with respect to the orientation of $S$, after isotoping the arcs to minimize intersections.

\begin{theorem}[Kazez--Roberts \cite{Kazez_Roberts}]
\label{Kazez_Roberts_estimate}
Let $(S,\phi)$ be an open book, and let $B$ be a boundary component of $S$ such that $\phi$ is right-veering at $B$.  Let $\alpha$ be an arc properly embedded in $S$, with at least one endpoint on $B$.  After isotoping $\alpha$ and $\phi(\alpha)$ to minimize intersections, let $i_{\phi}(\alpha)$ denote the signed count of intersections $p$ in the interiors of $\phi(\alpha)$ and $\alpha$ such that the union of the arc segments of $\alpha$ and $\phi(\alpha)$ from $B$ to $p$ are contained in an annular neighborhood of $B$. Then the fractional Dehn twist coefficient $c(\phi,B) \geq i_{\phi}(\alpha).$
\end{theorem}

Calculating fractional Dehn twist coefficients on a planar open book allows us to conclude facts about the supported contact structure, by the following result.

\begin{theorem}[Ito--Kawamuro \cite{Ito_Kawamuro_tight}]
\label{FDTC > 1}
If $(S,\phi)$ is an open book such that $S$ is a planar surface, and $c(\phi,B) > 1$ for every boundary component $B$ of $S$, then $(S,\phi)$ supports a tight contact structure.
\end{theorem}

In addition, if $\phi$ is pseudo-Anosov, we can conclude that the supported contact structure is universally tight.  Although the result was originally stated for connected binding, Baldwin and Etnyre \cite{Baldwin_Etnyre_transverse} noted that it works for multiple binding components as well.  Note that the formulation here is simpler than the actual theorem proved by Colin and Honda, but is sufficient for our needs.  

\begin{theorem}[Colin--Honda \cite{Colin_Honda}]
\label{Colin_Honda_universally_tight}
Let $(S,\phi)$ be an open book with pseudo-Anosov monodromy $\phi $. If $c(\phi,B) \geq 2$ for every boundary component $B$ of $S$, then the contact manifold supported by $(S,\phi)$ is universally tight.
\end{theorem}

\subsection{Hyperbolic Manifolds}

To prove Theorem~\ref{Th:hyperbolic_examples}, we wish to know how to construct hyperbolic manifolds.  We will use the following theorem of Ito and Kawamuro.

\begin{theorem}[Ito--Kawamuro \cite{Ito_Kawamuro_essential}]
\label{Ito_Kawamuro_hyperbolic}
If $(S,\phi)$ is an open book decomposition of $M$, and $c(\phi,B) > 4$ for every boundary component $B$ of $S$, then:
\begin{itemize}
\item $M$ is toroidal if and only if $\phi$ is reducible.
\item $M$ is hyperbolic if and only if $\phi$ is pseudo-Anosov.
\item $M$ is Seifert fibered if and only if $\phi$ is periodic.
\end{itemize}
\end{theorem}

To use Theorem~\ref{Ito_Kawamuro_hyperbolic}, we will construct pseudo-Anosov monodromies, using the following construction of Penner, based on work of Thurston.

\begin{theorem}[Penner \cite{Penner}]
\label{pseudo-Anosov}
Let $\Gamma_1$ and $\Gamma_2$ be two multicurves which fill a surface $S$ ({\it ie.\ }after being put in minimal position, their complement is a collection of discs and annuli), and let $\phi$ be a product of positive Dehn twists on the elements of $\Gamma_1$ and negative Dehn twists on the elements of $\Gamma_2$, where each curve in $\Gamma_1 \cup \Gamma_2$ appears at least once with non-zero exponent in $\phi$.  Then $\phi$ is pseudo-Anosov.
\end{theorem}

\subsection{Fillability}

When discussing $4$-manifold fillings $X$ of a $3$-manifold $M$, we will always assume that $\partial X = M$ with the same orientation.  For contact manifolds $(M,\xi)$, there are three main types of fillings: \textit{weak fillings}, \textit{strong fillings}, and \textit{Stein fillings}.

\begin{definition}
A contact manifold $(M,\xi)$ is \textit{weakly fillable} if $M$ is the oriented boundary of a symplectic manifold $(X,\omega)$ and $\omega|_{\xi} > 0$. 

It is \textit{strongly fillable} if in addition, there is a vector field $v$ pointing transversely out of $X$ along $M$ such that $\xi$ is isotopic to $\iota_v \omega$ and $\mathcal{L}_v \omega = \omega$.
\end{definition}

To define Stein fillability, we first define Stein domains.

\begin{definition}
A \textit{Stein domain} $(X,J,f)$ is an open complex manifold $(X,J)$ with a proper strictly plurisubharmonic function $f:X \rightarrow [0,\infty)$. In this case, $d(df \circ J) $ defines a symplectic form on $X$.  A contact manifold $(M,\xi)$ is \textit{Stein fillable} if $M$ is the regular level set of $f$ for some Stein domain $(X,J,f)$ and $\xi$ is isotopic to $TM \cap J(TM)$.
\end{definition}

If $(M,\xi)$ is Stein fillable, then $(M,\xi)$ is strongly fillable, and if $(M,\xi)$ is strongly fillable, then it is weakly fillable.  For planar contact manifolds, Niederkr\"{u}ger and Wendl have shown that the reverse implications hold.

\begin{theorem}[Niederkr\"{u}ger--Wendl \cite{Niederkruger_Wendl_Planar}]
\label{Niederkruger_Wendl}
If $(M,\xi)$ is a planar contact manifold, then the following are equivalent.
\begin{itemize}
\item $(M,\xi)$ is weakly fillable.
\item $(M,\xi)$ is strongly fillable.
\item $(M,\xi)$ is Stein fillable.
\end{itemize}
\end{theorem}

Eliashberg \cite{Eliashberg} and Gromov \cite{Gromov} showed that if $(M,\xi)$ is fillable, then $\xi$ is tight.

\subsection{Heegaard Floer contact invariant}

Another effective way to detect tightness is with tools from Heegaard Floer homology, defined by Ozsv\'{a}th and Szab\'{o} \cite{HF2,HF1}. We briefly recall the construction here. For any $3$-manifold $M$, Ozsv\'{a}th and Szab\'{o} defined a set of invariants, the simplest of which is the so-called \textit{hat} theory, which takes the form of a vector space $\widehat{HF}(M)$ over $\mathbb{Z}/2\mathbb{Z}$.  Given a contact structure $\xi$ on $M$, Ozsv\'{a}th and Szab\'{o} associate an element $c(\xi) \in \widehat{HF}(-M)$, see \cite{HF_contact_invariant}.  We call $c(\xi)$ the \textit{Heegaard Floer contact element}.

Ozsv\'{a}th and Szab\'{o} proved the following properties of the contact element.

\begin{itemize}
\item If $(M,\xi)$ is overtwisted, then $c(\xi) = 0$.
\item If $(M,\xi)$ is Stein fillable, then $c(\xi) \neq 0$.
\item If $(M',\xi')$ is obtained from $(M,\xi)$ by Legendrian surgery, \textit{ie.\ }surgery on a Legendrian knot with framing one less than the contact framing, then $c(\xi) \neq 0$ implies that $c(\xi') \neq 0$.
\end{itemize}

Theorem~\ref{Niederkruger_Wendl} allows us to conclude that in the case of planar contact manifolds, any of three fillability conditions (weak, strong, or Stein) implies that $c(\xi) \neq 0$.

Baldwin tracks the contact invariant under the capping off operation, and concludes the following.

\begin{theorem}[Baldwin~\cite{Baldwin_capping_off}]
\label{HF_capping_off}
Let $(S,\phi)$ be an open book decomposition of $(M,\xi)$, where $\partial S$ has more than one component.  Let $(S',\phi')$ be the open book obtained by capping off a boundary component of $S$, and let $(M',\xi')$ be the contact manifold it supports.  If $c(\xi) \neq 0$, then $c(\xi') \neq 0$.
\end{theorem}

\subsection{Transverse Surgery}
Given a framed transverse knot $K$ in a contact manifold $(M,\xi)$, a neighborhood of $K$ is contactomorphic to a standard neighborhood, \textit{ie.\ }the solid torus neighborhood $S_{r_0} = \{r \leq r_0\}$ of the $z$-axis in $\mathbb{R}^3/(z \sim z+1)$ with contact structure $\ker(\cos r\,dz + r\sin r\,\textrm{d}\theta)$, where $K$ gets mapped to the $z$-axis.  The boundary $\partial S_{r_0}$ has a characteristic foliation given by parallel linear leaves, of slope $-\cot r_0/r_0$ in the co-ordinate system on $\partial S_{r_0}$ given by $\left(\frac{\partial}{\partial z},\frac{\partial}{\partial\theta}\right)$.  The following exposition of surgery on transverse knots is due to Baldwin and Etnyre \cite{Baldwin_Etnyre_transverse}, based off of work of Martinet \cite{Martinet}, Lutz \cite{Lutz1, Lutz2}, and Gay \cite{Gay_transverse}.

\begin{definition} Given a neighborhood $N$ of $K$ contactomorphic to $S_{r_0}$, let $0 \leq r_1 < r_0$ be such that $-\cot r_1/r_1$ is a rational number $p/q$ ($q$ may be $0$).  Removing the interior of $S_{r_1}$ from $S_{r_0}$, we are left with a manifold where the characteristic foliation on the boundary is by curves of slope $p/q$.  We quotient $\partial (M \backslash S_{r_1})$ by the $S^1$-action of translation along the leaves to get $M'$.  It can be shown that this is a manifold, and that the contact structure on $M\backslash S_{r_1}$ descends to $M'$ in a well-defined manner.  We define \textit{admissible transverse $p/q$-surgery} on $K$ to be $(M',\xi')$, where $\xi'$ is the induced contact structure.  Notice that $M'$ is topologically $p/q$-surgery on $K$, with respect to the framing of $K$.
\end{definition}

\begin{remark} In general, this construction will depend on the choice of neighborhood of $K$.  In addition, there are infinitely many $r_1$ corresponding to the rational number $p/q$, although a given neighborhood of $K$ will contain only finitely many of them.  In all cases under discussion here, the neighborhood used and the particular choice of $r_1$ will be clear from the situation.
\end{remark}

A transverse knot $K$ in the binding of an open book $(S,\phi)$ has a framing induced by the page, called the \textit{page framing}. Baldwin and Etnyre show \cite{Baldwin_Etnyre_transverse} that if $K$ is not the only boundary component of $S$, then there exists a standard neighborhood of $K$, where the slope of the characteristic foliation on the boundary is $\epsilon > 0$, measured with respect to the page framing.  They then show that capping off an open book is an admissible transverse surgery on the binding component being capped off.  Note that the proof in \cite{Baldwin_Etnyre_transverse} extends to allow non-intersecting neighborhoods of multiple boundary components.  The advantage of this extension is that now capping off multiple boundary components can be seen as admissible surgery on a link.

\begin{proposition}[Baldwin--Etnyre \cite{Baldwin_Etnyre_transverse}]
\label{prop:capping_off}
Let $(S,\phi)$ be an open book with binding components $B_1, \ldots, B_k$, where $k \geq 2$.  Then there exist pairwise disjoint standard neighborhoods $N_1, \ldots, N_{k-1}$ of $B_1, \ldots, B_{k-1}$, where $N_i$ is contactomorphic to $S_{\epsilon_i}$ for some $\epsilon_i > 0$, with respect to the page framing.
\end{proposition}

\section{Proofs of results}
\label{Section:Proofs}

We start with a construction that will be used throughout the proofs of our examples.

\begin{construction}\label{construction}
Let $(S,\phi)$ be an open book, and let $B$ be a boundary component of $S$, shown on the left side of Figure~\ref{fig:construction}.  We attach the surface shown on the right side of Figure~\ref{fig:construction}, by identifying the boundary components labeled $B$. The new surface $S'$ has replaced $B$ with $m$ boundary components, for some integer $m \geq 2$. We label the new boundary components by $B'_1,\ldots, B'_m$. The new open book decomposition has monodromy $\phi' = \phi \circ \tau_B^{-n_m}\tau_{B'_1}^{n_1}\cdots \tau_{B'_m}^{n_m}$, where $\phi$ is extended by the identity to be defined on all of $S'$. Note that after capping off the boundary components $B'_1, \ldots, B'_{m-1}$ we get the original open book decomposition $(S,\phi)$.  See Figure~\ref{fig:construction} for a graphical representation of the construction.
\end{construction}

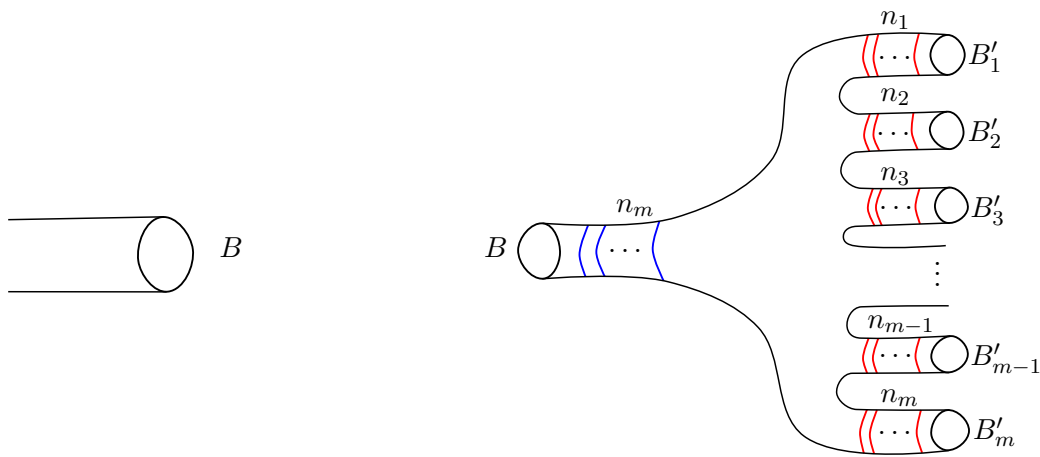
\begin{figure}[htb]
\begin{overpic}
{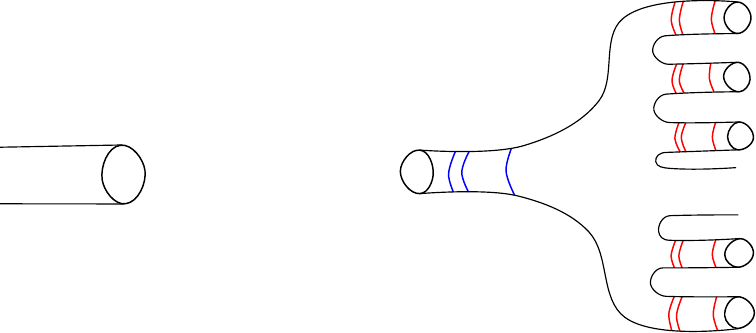}
\put(330,163){$n_1$}
\put(363,148){$B'_1$}
\put(330,134){$n_2$}
\put(363,119){$B'_2$}
\put(330,104){$n_3$}
\put(364,90){$B'_3$}
\put(325,48){$n_{m-1}$}
\put(364,34){$B'_{m-1}$}
\put(330,20){$n_{m}$}
\put(365,6){$B'_m$}
\put(230,92){$n_{m}$}
\put(80,75){$B$}
\put(180,75){$B$}
\put(227,77){\dots}
\put(329,150){\dots}
\put(328.5,121){\dots}
\put(329.5,93){\dots}
\put(329.3,37){\dots}
\put(329,8){\dots}
\put(350,64){$\vdots$}
\end{overpic}
\caption{Construction~\ref{construction} on a boundary component $B$.  The red curves on the right are positive Dehn twists, and the blue curves on the left, parallel to $B$, are negative Dehn twists.}
\label{fig:construction}
\end{figure}

\begin{lemma}
\label{lemma:fractional_dehn_twist_coefficient}
The fractional Dehn twist coefficient $c(\phi',B'_i) = n_i$ for each $i = 1, \ldots, m$.
\end{lemma}
\begin{proof}
Let $\gamma$ be a simple closed curve on $S'$ isotopic to $B$.  Then $\gamma$ is fixed by $\phi'$.  Thus $\phi'$ is a reducible monodromy, and to calculate the fractional Dehn twist coefficient at $B'_i$, we need to find the maximal subsurface $S''$ of $S'$ containing $B'_i$ on which $\phi'$ restricts to a pseudo-Anosov or periodic monodromy.  This subsurface can be identified with the surface added to $S$ to make $S'$, namely, the right-hand side of Figure~\ref{fig:construction}. Any larger subsurface would contain $B$ as a non-boundary-parallel curve, and the monodromy restricted to the larger subsurface would be reducible.

Now, restricted to $S''$, $\phi'$ is periodic, and is in fact already isotopic to a product of boundary Dehn twists.  Thus, the count of Dehn twists at each boundary component is the fractional Dehn twist coefficient, and the lemma follows.
\end{proof}

We now prove Theorem~\ref{Th:Main_Theorem}, which is a corollary of Construction~\ref{construction}.

\begin{proof}[Proof of Theorem~\ref{Th:Main_Theorem}]
Let $(M,\xi)$ be an overtwisted contact manifold. By Theorem~\ref{Etnyre OT planar}, $(M,\xi)$ is supported by a planar open book decomposition $(S,\phi)$.  Apply Construction~\ref{construction} to each boundary component of $S$, with each $m \geq 2$ and each $n_i \geq 2$, to get another planar open book decomposition $(S',\phi')$.  Denote the contact manifold supported by this open book decomposition by $(M',\xi')$.  Since the fractional Dehn twists coefficient $c(\phi',B') \geq 2$ for each boundary component $B'$ of $S'$ by Lemma~\ref{lemma:fractional_dehn_twist_coefficient}, Theorem~\ref{FDTC > 1} implies that $\xi'$ is tight.  Note that after capping off all but one boundary component on each added surface, we recover $(S,\phi)$. Since $(S,\phi)$ is overtwisted, the Heegaard Floer contact invariant $c(\xi) = 0$, so Theorem~\ref{HF_capping_off} implies that the contact invariant $c(\xi') = 0$.  Since $S'$ is planar, Theorem~\ref{Niederkruger_Wendl} implies that $(M',\xi')$ is not weakly, strongly, or Stein fillable.  Let $L$ be the transverse link given by the boundary components of $S'$ capped off to recover $(S,\phi)$.  Proposition~\ref{prop:capping_off} gives us a neighborhood of $L$ such that we can realize the capping off operation as admissible surgery on each component of the link.  Thus, admissible surgery on $L$ in $(M',\xi')$ will recover $(M,\xi)$.  Finally, note that we have infinitely many choices of $(M',\xi')$, as for each boundary component of $S'$, we have infinitely many choices of $m$ and $n_i$.
\end{proof}

We now turn to the proof of Theorem~\ref{Th:hyperbolic_examples}.

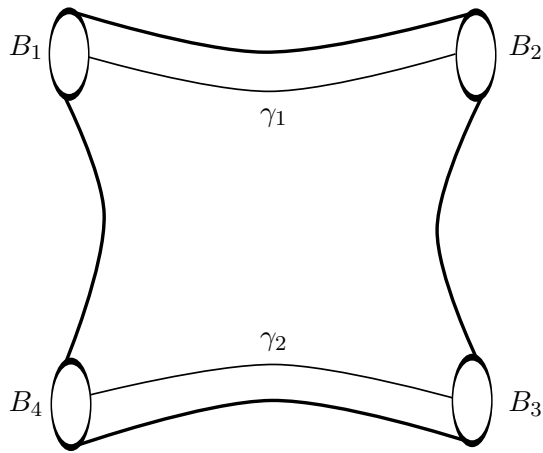
\begin{figure}[htb]
\begin{overpic}[scale=.8,tics=20]
{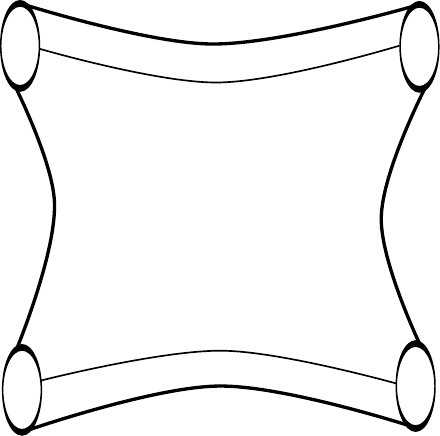}
\put(80,125){$\gamma_1$}
\put(80,40){$\gamma_2$}
\put(-15,150){$B_1$}
\put(174,150){$B_2$}
\put(-15,15){$B_4$}
\put(174,15){$B_3$}
\end{overpic}
\caption{$S$ with $\gamma_1$ and $\gamma_2$ indicated.}
\label{fig:arcs_on_surface}
\end{figure}

\begin{proof}[Proof of Theorem~\ref{Th:hyperbolic_examples}]
First note that if $\Gamma_1=\{\alpha\}$ and $\Gamma_2 = \{\beta\}$, where $\alpha$ and $\beta$ are as in Figure~\ref{fig:Planar_Surface_With_the_Curves_Drawn}, then $\Gamma_1$ and $\Gamma_2$ are multicurves that fill the surface $S$.  Thus by Theorem~\ref{pseudo-Anosov}, the diffeomorphism $\phi' = \tau_\alpha^{-n_1-1} \tau_\beta^p$ is pseudo-Anosov, as $p \geq 1$ and $n_1 \geq 6$.  Similarly, composing $\phi'$ with a product of boundary parallel Dehn twists does not change the free isotopy class of $\phi'$, thus $\phi_{\mathbf{v}}$ is also pseudo-Anosov.  Let $\gamma_1$ and $\gamma_2$ be the arcs shown in Figure~\ref{fig:arcs_on_surface}.  One can check that $i_{\phi_{\mathbf{v}}}(\gamma_1)$ near $B_1$ and $B_2$ is $n_1 - 1$ and $n_2 - 1$ respectively, and similarly for $i_{\phi_{\mathbf{v}}}(\gamma_2)$ near $B_3$ and $B_4$.  Thus, by Theorem~\ref{Kazez_Roberts_estimate}, the fractional Dehn twist coefficient $c(\phi,B_i) \geq n_i - 1 \geq 5 > 4$.  By Theorem~\ref{Ito_Kawamuro_hyperbolic}, the contact manifold $(M_\mathbf{v},\xi_{\mathbf{v}})$ supported by $(S,\phi_{\mathbf{v}})$ is hyperbolic.  By Theorem~\ref{Colin_Honda_universally_tight}, $\xi_{\mathbf{v}}$ is universally tight.  Note that capping off $B_2$ results in an open book supporting an overtwisted contact structure, as the monodromy around $B_1$ consists of a negative Dehn twist.  Thus, similarly to the proof of Theorem~\ref{Th:Main_Theorem}, $(M_{\mathbf{v}},\xi_{\mathbf{v}})$ has $c(\xi_{\mathbf{v}}) = 0$ and is not fillable, by Theorem~\ref{HF_capping_off} and Theorem~\ref{Niederkruger_Wendl}.

We now want to show that there are infinitely many distinct examples as we range over various values of $\mathbf{v}$. To do this, we compute the order of the group $H_1(M_{\mathbf{v}};\mathbb{Z})$ and show that it attains infinitely many values.

\begin{figure}[htb]
\begin{overpic}[scale=5,tics=20]
{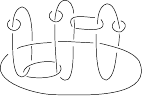}
\put(78,180){$0$}
\put(179,207){$0$}
\put(281,200){$0$}
\put(4,178){$B_2$}
\put(3,160){$-1$}
\put(103,190){$B_3$}
\put(103,175){$-1$}
\put(300,178){$B_4$}
\put(305,160){$-1$}
\put(200,6){$B_1$}
\put(170,5){$-1$}
\put(200,193){$\alpha$}
\put(200,178){$1$}
\put(100,85){$\beta$}
\put(100,70){$-1$}
\end{overpic}
\caption{The knots involved in a surgery diagram for $M_{\mathbf{v}}$.}
\label{fig:kirby_diagram}
\end{figure}
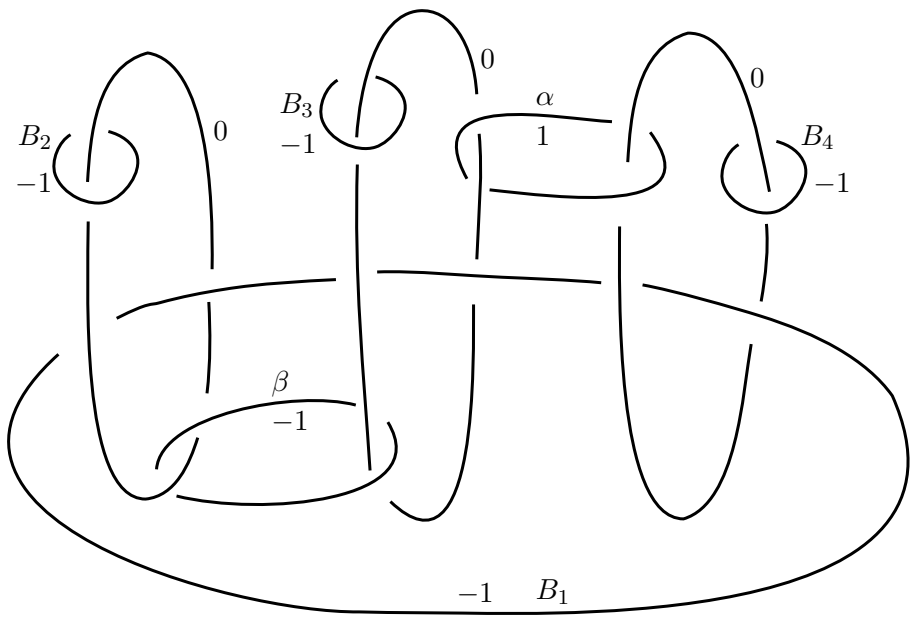

To get a surgery diagram for $M_{\mathbf{v}}$, we use the following fact: a simple closed curve on the page of an open book corresponds to a knot in the $3$-manifold described by the open book, and a right-handed (resp.\ left-handed) Dehn twist about this curve corresponds to a surgery on the knot with framing $-1$ (resp.\ $1$), see \cite[Theorem~5.7]{Etnyre_Open_Book_Decomposition}.

With this information, we describe a surgery diagram for $M_{\mathbf{v}}$ using Figure~\ref{fig:kirby_diagram}. The $0$-framed unknots come from treating the page as a disc bounded by $B_1$ with three $1$-handles attached.  The surgeries on the knots labeled $B_1, \ldots, B_4$ correspond to the boundary-parallel Dehn twists in the monodromy; take $n_i$ copies of $B_i$, and do surgery with framing $-1$ on each. The knots $\alpha$ and $\beta$ correspond to the $\alpha$ and $\beta$ curves on the page of the open-book (see Figure~\ref{fig:Planar_Surface_With_the_Curves_Drawn}); take $n_1+1$ copies of $\alpha$ and $p$ copies of $\beta$, and do surgery with framing $1$ on the copies of $\alpha$ and with framing $-1$ on the copies of $\beta$. All copies are push-offs using the Seifert framing of the knot in $S^3$, and the surgery framings are with respect to this same Seifert framing.

Now, after blowing down the $\pm 1$ framed components, we arrive at a surgery diagram which is surgery on a three-component link, and where the linking matrix is given by
$$
L_{\mathbf{v}}=\left(
\begin{array}{ccc}
n_1+n_2+p  & n_1+p  & n_1  \\
n_1+p  & n_3+p-1  & -1  \\
n_1  &  -1 & n_4-1  
\end{array}
\right).
$$
The order of $H_1(M_{\mathbf{v}};\mathbb{Z})$ is given by the absolute value of the determinant of the linking matrix $L_{\mathbf{v}}$.  Indeed, if we let $$\mathbf{v}_k = (6,11,6,6,24+k),$$ the order of $H_1(M_{\mathbf{v}_k};\mathbb{Z})$ is $k$ (where $k=0$ means that $H_1(M_{\mathbf{v}_0};\mathbb{Z})$ is infinite).
\end{proof}

\bibliography{bibliography}
\bibliographystyle{plain}

\end{document}